\documentclass{article}

\usepackage{amsmath}
\usepackage{amsfonts}
\usepackage{amsthm}
\usepackage{amssymb}
\usepackage[english]{babel}
\usepackage[ansinew]{inputenc}
\usepackage[all]{xy}

\newtheorem{defi}{Definition}[section]
\newtheorem{teo}[defi]{Theorem}
\newtheorem{cor}[defi]{Corollary}
\newtheorem{lem}[defi]{Lemma}
\newtheorem{prop}[defi]{Proposition}
\newtheorem{rem}[defi]{Remark}
\newtheorem{nota}[defi]{Notation}

\newtheorem{defin}[defi]{Definition}

\author{Marco Antei}

\title{On the Abelian Fundamental Group Scheme of a Family of Varieties}
\begin{document}
\maketitle

\noindent \textbf{Abstract.} Let $S$ be a connected Dedekind scheme and  $X$ an $S$-scheme provided with a section $x$. We prove that the morphism between fundamental group schemes $\pi_1(X,x)^{ab}\to  \pi_1(\mathbf{Alb}_{X/S},0_{\mathbf{Alb}_{X/S}})$ induced by the canonical morphism from $X$ to its Albanese scheme $\mathbf{Alb}_{X/S}$ (when the latter exists) fits in an exact sequence of group schemes $0\to  (\mathbf{NS}^{\tau}_{X/S})^{\vee}\to  \pi_1(X,x)^{ab}\to \pi_1(\mathbf{Alb}_{X/S},0_{\mathbf{Alb}_{X/S}}) \to  0$ where the kernel is a finite and flat $S$-group scheme. Furthermore we prove  that any finite and commutative quotient pointed torsor over the generic fiber $X_{\eta}$ of $X$ can be extended to a finite and commutative pointed torsor over $X$. 
\medskip
\\\indent \textbf{Mathematics Subject Classification}: Primary: 14H30, 14K30, 14L15. Secondary: 11G99.\\\indent
\textbf{Key words}: torsors, fundamental group scheme, Jacobian, Albanese scheme.

\tableofcontents
\bigskip

\section{Introduction}\label{sez:Intro} A classical result states that the abelianized \'etale fundamental group of a complete smooth curve  over a separably closed field is isomorphic to the \'etale fundamental group of its Jacobian (cf. \cite{Mil2}, \S 9).  In this paper we will generalize this result in the language of the fundamental group scheme introduced by Nori (cf. \cite{Nor} and \cite{Nor2}) for schemes over fields then extended by Gasbarri (cf. \cite{Gas}) for schemes over Dedekind schemes. In particular in \textbf{Corollary \ref{corCurve}} we prove  that if $f:C\to S$ is a smooth and projective curve with integral geometric fibers endowed with a $S$-valued point $x\in C(S)$ then the natural morphism $\varphi:\pi_1(C,x)^{ab}\to \pi_1(J,0_J)$ from the abelianized fundamental group scheme of $C$ to the fundamental group scheme of its Jacobian is an isomorphism. This is a consequence of a more general statement (cf. \textbf{Theorem \ref{teoUnico}}) in higher dimension: if we replace $C$ by a scheme of finite type $X$ and $J$ by $\mathbf{Alb}_{X/S}$, the Albanese scheme of $X$  (provided it exists), then the morphism of fundamental group schemes $\pi_1(X,x)^{ab}\to  \pi_1(\mathbf{Alb}_{X/S},0_{\mathbf{Alb}_{X/S}})$ induced by the canonical morphism $X\to \mathbf{Alb}_{X/S}$  fits in an exact sequence of group schemes $0\to  (\mathbf{NS}^{\tau}_{X/S})^{\vee}\to  \pi_1(X,x)^{ab}\to \pi_1(\mathbf{Alb}_{X/S},0_{\mathbf{Alb}_{X/S}}) \to  0$ where the kernel is  the Cartier dual of a $S$-finite and smooth group scheme $\mathbf{NS}_{X/S}^{\tau}$, called torsion N\'eron-Severi scheme of $X$ over $S$;  when $S$ is the spectrum of an algebraically closed field $k$ of characteristic $0$ this coincides with a classical statement (cf. \cite{Mil}, III, \S 4, Corollary 4.19 and \cite{Szamuely}, \S 5.8) whose main techniques are used here to solve our problem.\\

 As already discussed in \cite{Antei} the study of the fundamental group scheme is tightly related to the problem of extending finite and pointed torsors over the generic fiber $X_{\eta}$ of $X$ to torsors over $X$. Here we concentrate our attention on commutative torsors. We already know that if $X$ is an abelian scheme every (necessarily) commutative quotient (i.e. the group scheme acting on it is a quotient of the fundamental group scheme of $X_{\eta}$) pointed torsor over $X_{\eta}$ can be extended to $X$ (cf. \cite{Antei}, \S 3.2). Tossici recently proved in \cite{Tos}, Corollary 4.9 that if $S=Spec(R)$, where $R$ is a d.v.r. of mixed characteristic and $X$ is a normal scheme, faithfully flat over $S$ with integral fibers then every commutative finite torsor (pointed or not) $Y'\to X_{\eta}$ with $Y'$ connected such that the normalization $Y$ of $X$ in $Y'$ has integral special fiber can be extended to a commutative finite torsor over $X$ up to an extension of $R$. Here we prove (cf. \textbf{Theorem \ref{teoEstensioneTorsoriCommutativi}}) that for any Dedekind scheme $S$, under some existence assumptions for $\mathbf{Pic}_{X/S}^{0}$ and for $\mathbf{Pic}_{X/S}^{\tau}$,  every commutative finite pointed torsor  $Y'\to X_{\eta}$ can be extended to a commutative finite pointed torsor  over $X$. We do not need to extend the base scheme $S$ and we do not need other assumptions on $Y'$, but of course we may have problems while trying to solve existence problems for $\mathbf{Pic}_{X/S}^{0}$ and $\mathbf{Pic}_{X/S}^{\tau}$. In section \ref{sez:Effective} we provide some examples where everything behaves well.

\begin{nota} Throughout this paper every scheme is supposed to be locally noetherian. We denote by $Sch/S$ the category of (locally noetherian) schemes over a base scheme $S$. If $X$ is a $S$-scheme, a torsor over $X$ under a finite and flat $S$-group scheme $G$ will be a finite, faithfully flat and $G$-invariant morphism $p: Y \to X$, locally trivial for the fpqc topology, where $Y$ is a $S$-scheme endowed with a right action of $G$.\end{nota}
\indent \textbf{Acknowledgements}
I would like to thank the Max-Planck-Institut f\"ur Mathematik, l'Universit\'e de Lille 1 and the Isaac Newton Institute of Cambridge for giving me the opportunity to work in very dynamical and stimulating environments. I also would like to thank Vikram Mehta for suggesting me the problem in the case of curves, Dajano Tossici for very useful suggestions and discussions and Michel Emsalem for his constant (not only) mathematical help. 

\section{Preliminaries}
\subsection{The Picard scheme}
\label{sez:Picard}
In this section we briefly recall the definitions and properties of the Picard scheme and related objects that we will need in the next sections. The following material is mostly taken from Grothendieck's seminars \cite{GroBourV} and \cite{GroBourVI} and the more recent Kleiman's exposition \cite{KL} (also available on the arXiv). Throughout this section $X$ and $S$ will be schemes and $f:X\to S$ a separated morphism of finite type. For any base change $T\to S$ we denote for convenience $X_T:=X\times_S T$ and $f_T:X_T\to T$. As usual $Pic(X)\simeq H^1(X,\mathcal{O}_X^{\ast})$ will denote the absolute Picard group of $X$, i.e. the group of isomorphism classes of invertible sheaves of $\mathcal{O}_X$-modules. Now consider the two contravariant functors $Pic_{X}$ and $Pic_{X/S}$ from the category $\mathcal{S}ch/S$  to the category of abelian groups $\mathcal{A}b$ given by formulas 

$$  Pic_X(T):=Pic(X_T),$$

$$  Pic_{X/S}(T):=Pic(X_T)/Pic(T).$$
 They are called, respectively, the absolute Picard functor and the relative Picard functor. Then denote by $Pic_{(X/S)(zar)}$, $Pic_{(X/S)(\acute{e}t)}$, $Pic_{(X/S)(fppf)}$ the sheaves in the Zariski, étale and fppf topology associated to $Pic_{X/S}$ (or, equivalently, to $Pic_X$). When $S$ is the spectrum of a commutative ring $R$ with unity  we will usually write ${Pic}_{X/R}$ instead of ${Pic}_{X/Spec(R)}$. 
 
 \begin{defin} We will say that 
$\mathcal{O}_S\simeq f_{\ast}\mathcal{O}_X$ holds universally if for any $S$-scheme $T$ the natural morphism  $\mathcal{O}_T\to {f_T}_{\ast}\mathcal{O}_{X_T}$ is an isomorphism. 
\end{defin}

\begin{rem}\label{ossUniversal} When $f:X\to S$ is proper and flat with reduced and connected geometric fibers then $\mathcal{O}_S\simeq f_{\ast}\mathcal{O}_X$ holds universally (cf. \cite{KL}, Exercise 9.3.11).
\end{rem}

\begin{teo}\label{teoKleiman} Assume $\mathcal{O}_S\simeq f_{\ast}\mathcal{O}_X$ holds universally. Assume moreover that $X(S)\neq \emptyset$, then
the natural maps 
$$Pic_{X/S}\to Pic_{(X/S)(zar)}\to Pic_{(X/S)(\acute{e}t)}\to Pic_{(X/S)(fppf)}$$
 are isomorphisms.
\end{teo}
\proof 
Cf. \cite{KL}, Theorem 9.2.5. \endproof

\begin{defin} If any of the functors $Pic_{X/S}, Pic_{(X/S)(zar)}, Pic_{(X/S)(\acute{e}t)}$ or $Pic_{(X/S)(fppf)}$ is
representable, then the representing scheme is called the Picard scheme and denoted
by $\mathbf{Pic}_{X/S}$.
\end{defin}

Now we recall an existence theorem for the Picard scheme over a general base $S$; a sharper result when $S$ is the spectrum of a field has been proven by Murre and Oort (cf. \cite{Murre}, II.15, Theorem 2).

\begin{teo}\label{teoPicard} Assume that $f:X\to S$ is  flat and projective with integral geometric fibers, then $\mathbf{Pic}_{X/S}$ exists, is separated and locally of finite type over $S$ and represents $Pic_{(X/S)(\acute{e}t)}$ (and $Pic_{(X/S)(fppf)}$). Assume moreover that $X(S)\neq \emptyset$ then $\mathbf{Pic}_{X/S}$ represents all the functors $Pic_{X/S}$, $Pic_{(X/S)(zar)}$, $Pic_{(X/S)(\acute{e}t)}$ and $Pic_{(X/S)(fppf)}$. 
\end{teo}

\proof By \cite{KL}, Theorem 9.4.8 the Picard scheme $\mathbf{Pic}_{X/S}$ exists, is separated and locally of finite type over $S$ and represents $Pic_{(X/S)(\acute{e}t)}$ (and consequently $Pic_{(X/S)(fppf)}$). Then use remark \ref{ossUniversal} and theorem \ref{teoKleiman} for the second assertion. \endproof

\begin{rem} When $\mathbf{Pic}_{X/S}$ exists then for any base change $S'\to S$ the Picard scheme  $\mathbf{Pic}_{X_{S'}/S'}$ also exists and $\mathbf{Pic}_{X_{S'}/S'}\simeq \mathbf{Pic}_{X/S}\times_S S'$ (cf. \cite{KL}, Ex. 9.4.4).\end{rem} 

We recall that $X$ is an abelian scheme over $S$ of relative dimension $g$ if $X$ is a smooth and proper $S$-group scheme with  connected geometric fibers of dimension $g$. An abelian scheme is always a commutative group scheme (cf. \cite{mumGIT}, Ch. 6, \S 1, Corollary 6.5).

\begin{rem}\label{remPicard} When $X\to S$ is an abelian projective scheme then $\mathcal{O}_S\simeq f_{\ast}\mathcal{O}_X$ holds universally and of course $X(S)\neq \emptyset$ so theorem \ref{teoPicard} holds.\end{rem}

Let $X$ be a scheme over a field $k$ such that $\mathbf{Pic}_{X/k}$ exists, we denote by $\mathbf{Pic}_{X/k}^0$ the connected component of the Picard scheme $\mathbf{Pic}_{X/k}$ containing the unity. It is a geometrically irreducible $k$-group scheme of finite type (cf. \cite{KL}, Lemma 9.5.1). For a general base scheme  $S$ we denote by $Pic^0_{X/S}$ the union, as a set, of all the $\mathbf{Pic}_{X_s/k(s)}^0$; assume the existence of $\mathbf{Pic}_{X/S}$ and suppose there exists an open $S$-group subscheme of $\mathbf{Pic}_{X/S}$, that we denote by $\mathbf{Pic}_{X/S}^{0}$, satisfying the property that for any point $s\in S$ 

$$ \mathbf{Pic}_{X/S}^{0}\times_S k(s) \simeq \mathbf{Pic}_{X_s/k(s)}^0.$$ As a set it is just $Pic^0_{X/S}$; by definition it has irreducible geometric fibers so if moreover it is smooth and proper over $S$ it is an abelian scheme. We discuss the existence of $\mathbf{Pic}_{X/S}^{0}$ in the following theorem:

\begin{teo}\label{teoesistenza1}
Assume  that $\mathbf{Pic}_{X/S}$ exists, is separated over $S$  and represents $Pic_{(X/S)(fppf)}$. For
any $s\in S$ assume all the $\mathbf{Pic}_{X_s/{k(s)}}^0$
 are smooth of the same dimension. Then $\mathbf{Pic}_{X/S}^{0}$ exists and is of finite type over $S$.
Furthermore, if $S$ is reduced, then $\mathbf{Pic}_{X/S}^{0}$ is smooth over $S$. Moreover, if $S$ is noetherian then $\mathbf{Pic}_{X/S}^{0}$ is closed in
$\mathbf{Pic}_{X/S}$ and projective over $S$.
\end{teo}
\begin{proof} By \cite{KL}, Proposition 9.5.20 $\mathbf{Pic}_{X/S}^{0}$ exists and is of finite type over $S$ and it is smooth over $S$ when $S$ is reduced. By \cite{KL}, Ex. 9.5.7 all the $\mathbf{Pic}_{X_s/{k(s)}}^0$ are projective over $k(s)$ then $\mathbf{Pic}_{X/S}^{0}$ is closed in $\mathbf{Pic}_{X/S}$ by \cite{KL}, Proposition 9.5.20. Again by \cite{KL}, Ex. 9.5.7 $\mathbf{Pic}_{X/S}^{0}$ is projective over $S$.
\end{proof}

\begin{rem}\label{ossPicCarZero} If $char(k(s))=0$ then $\mathbf{Pic}_{X_s/{k(s)}}^0$
 is smooth over $k(s)$ by \cite{mumLec}, Lecture 25, Theorem 1, and if $char(k(s))=0$ for all $s\in S$ then all the $\mathbf{Pic}_{X_s/{k(s)}}^0$ are smooth of the same dimension (cf. \cite{KL}, Remark 9.5.21) so in this case we don't need to add this assumption in the statement of theorem \ref{teoesistenza1}.
\end{rem}

 When $\mathbf{Pic}_{X/S}$ exists let $n:\mathbf{Pic}_{X/S}\to \mathbf{Pic}_{X/S}$ be the multiplication by $n$ for every integer $n>0$;  we define the set $$\mathbf{Pic}_{X/S}^{\tau}:=\cup_n n^{-1}({Pic}_{X/S}^0).$$

\noindent If it is open in $\mathbf{Pic}_{X/S}$ it inherits a structure of scheme. We have the following

\begin{teo}\label{teoesistenza2} Assume $S$ is noetherian, $f:X\to S$ is projective and flat with irreducible geometric fibers. Then $\mathbf{Pic}_{X/S}^{\tau}$ is an open and closed group subscheme of $\mathbf{Pic}_{X/S}$, quasi projective and of finite type over $S$. If moreover $f$ is smooth then $\mathbf{Pic}_{X/S}^{\tau}$ is projective over $S$.
\end{teo}
\begin{proof}
By theorem \ref{teoPicard}  $\mathbf{Pic}_{X/S}$ exists, is separated and locally of finite type over $S$. Then use \cite{KL}, Theorem 9.6.16 and Ex. 9.6.18.
\end{proof}
In particular $\mathbf{Pic}_{X/S}^{\tau}$ is flat (resp. smooth) if $\mathbf{Pic}_{X/S}$ is flat (resp. smooth).

\begin{prop}\label{propDualAbel} Let $f:X\to S$ be a projective abelian scheme. Then $X^{\ast}:=\mathbf{Pic}_{X/S}^{0}$ exists, is a projective abelian scheme over $S$ and $\mathbf{Pic}_{X/S}^{0}=\mathbf{Pic}_{X/S}^{\tau}$. It is called the dual abelian scheme of $X$.
\end{prop}\proof Cf. \cite{mumGIT}, Corollary 6.8. and \cite{Oort}, I.5, Property (5.3).\endproof

\begin{defin}\label{defAlbanese} When  $\mathbf{Pic}_{X/S}^0$ exists and  is a projective abelian scheme we call it the canonical abelian subscheme of $\mathbf{Pic}_{X/S}$. Set $\mathbf{Alb}_{X/S}:={(\mathbf{Pic}_{X/S}^0)}^{\ast}=\mathbf{Pic}_{{\mathbf{Pic}_{X/S}^{0}}/S}^{0}$ which is a projective abelian scheme over $S$ according to proposition \ref{propDualAbel}. We call it the Albanese scheme of $X\to S$.
\end{defin}

\begin{prop}\label{propAlbanese} Assume $f:X\to S$ has connected fibres, $\mathcal{O}_S\simeq f_{\ast}\mathcal{O}_X$ holds universally and let $x\in X(S)$. Assume $\mathbf{Pic}_{X/S}^0$ exists and  is a projective abelian scheme. Then there exists a canonical morphism $$\lambda:X\to \mathbf{Alb}_{X/S}$$ sending the $S$-valued section $x$ to $0_{\mathbf{Alb}_{X/S}}$. Furthermore forming $\mathbf{Alb}_{X/S}$ commutes with changing the base.
\end{prop}
\proof The second assertion follows from \cite{GroBourVI}, Th\'eor\`eme 3.3 (iii). For the existence of $\lambda$ we first observe that by theorem \ref{teoKleiman} $\mathbf{Pic}_{X/S}$ represents $Pic_{X/S}$ thus Yoneda's lemma (cf. \cite{Vis} Proposition 2.3) assure the existence of a unique universal sheaf $\mathcal{P}\in Pic(X\times \mathbf{Pic}_{X/S})$ called the Poincaré sheaf (see \cite{KL}, Exercice 9.4.3 for  more details). Consider $\mathcal{L}:=(1_X\times i)^{\ast}(\mathcal{P})$ where $i:\mathbf{Pic}_{X/S}^0\hookrightarrow \mathbf{Pic}_{X/S}$ is the inclusion map, then consider $$\overline{\mathcal{L}}:=(\mathcal{L}\text{ mod }Pic(X))\in Pic(X\times \mathbf{Pic}_{X/S}^0)/Pic(X)={\mathbf{Pic}}_{\mathbf{Pic}_{X/S}^0}(X)$$ thus to $\overline{\mathcal{L}}$ we associate a morphism $\sigma:X\to {\mathbf{Pic}}_{\mathbf{Pic}_{X/S}^0}$. Now we eventually make a translation by considering $\lambda:=(\sigma-\sigma\circ x \circ f):X\to {\mathbf{Pic}}_{\mathbf{Pic}_{X/S}^0}$ observing that $\lambda\circ x:S\to {\mathbf{Pic}}_{\mathbf{Pic}_{X/S}^0}$ is the zero map thus $\lambda$ sends the $S$-valued section $x$ to the identity element of ${\mathbf{Pic}}_{\mathbf{Pic}_{X/S}^0}$. Since $X$ has connected fibers then $\lambda$ factors through ${\mathbf{Pic}}_{\mathbf{Pic}_{X/S}^0}^{0}=\mathbf{Alb}_{X/S}$ as required.\endproof

\begin{teo}\label{teoDualIsom} Let $f:X\to S$ be a projective abelian scheme. The canonical morphism $\lambda:X\to X^{\ast\ast}$ defined in proposition \ref{propAlbanese} is an isomorphism of group schemes.
\end{teo}\proof Cf. \cite{Oort}, III.20, Theorem (20.2) or \cite{KL}, Remark 9.5.24.\endproof

\begin{rem}\label{ossInclusioniPicchi} From the morphism $\lambda:X\to \mathbf{Alb}_{X/S}$ defined in \ref{propAlbanese} one deduces, taking the pull back, a $S$-group scheme morphism  $\lambda^{\ast}:\mathbf{Pic}_{\mathbf{Alb}_{X/S}/S}\to \mathbf{Pic}_{X/S}$ thus, composing with the inclusion, a morphism $\gamma:\mathbf{Pic}^{0}_{\mathbf{Alb}_{X/S}/S}\to \mathbf{Pic}_{X/S}$ which  factors through $\mathbf{Pic}_{X/S}^{0}$. The resulting morphism $$\gamma:\mathbf{Pic}^{0}_{\mathbf{Alb}_{X/S}/S}\to \mathbf{Pic}^0_{X/S}$$ is the isomorphism of theorem \ref{teoDualIsom} (see also \cite{Szamuely} Facts 5.8.9 for an informal discussion and references). From now on we identify $\mathbf{Pic}^0_{X/S}$ and $\mathbf{Pic}^{0}_{\mathbf{Alb}_{X/S}/S}$  through the isomorphism $\gamma$.
\end{rem}

Now we recall a duality  theorem concerning abelian schemes:

\begin{teo}\label{teoDualSequence} Let $X$ and $Y$ be projective abelian schemes over $S$. Let $\varphi:X\to Y$ be an isogeny and let $$\xymatrix{0\ar[r] & N \ar[r]& X \ar[r]^{\varphi} & Y \ar[r] & 0}$$ the corresponding exact sequence where $N:=ker(\varphi)$ is a $S$-finite and  flat group scheme. Then there exists an exact sequence
$$\xymatrix{0\ar[r] & N^{\vee} \ar[r]& Y^{\ast} \ar[r]^{\varphi^{\ast}} & X^{\ast} \ar[r] & 0}$$ where $N^{\vee}$ is the Cartier dual of $N$.
\end{teo}\proof Cf. \cite{Oort}, III.19, Theorem (19.1).\endproof

In next theorem we analyze the case of curves:

\begin{teo}\label{teoPicardMumf} Let $f:X\to S$ be a smooth and projective curve with integral geometric fibers. Then $J_{X/S}:=\mathbf{Pic}_{X/S}^{0}=\mathbf{Pic}_{X/S}^{\tau}$ is a projective abelian scheme and there is a canonical isomorphism $$\theta:J_{X/S}\to J_{X/S}^{\ast}.$$\end{teo}
\proof This follows from \cite{mumGIT}, Ch. 6, Proposition 6.9. \endproof

\begin{defin}\label{defJacobian} The scheme $J_{X/S}$ as in theorem \ref{teoPicardMumf} is called the Jacobian (scheme) of the curve $X$.
\end{defin}

\begin{rem} Under the same assumptions of theorem \ref{teoPicardMumf} it is clear that $J_{X/S}\simeq \mathbf{Alb}_{X/S}$. If moreover the genus $g=1$ then $J_{X/S}\simeq X$ canonically by theorem \ref{teoDualIsom}.
\end{rem}

\subsection{The torsion N\'eron-Severi scheme}
\label{sez:Neron}
Even if in some relevant cases $\mathbf{Pic}_{X/S}^{\tau}=\mathbf{Pic}_{X/S}^{0}$ (e.g. relative curves, cf. theorem \ref{teoPicardMumf} and abelian schemes, cf. proposition \ref{propDualAbel}), it is not always the case. In general, by definition, $\mathbf{Pic}_{X/S}^{\tau}$ contains $\mathbf{Pic}_{X/S}^{0}$. Throughout this section we keep the following 

\begin{nota}\label{notaNS} We assume the existence of $\mathbf{Pic}_{X/S}$ instead of specifying hypothesis on the morphism $f:X\to S$ implying it exists and we also assume the existence of both $\mathbf{Pic}_{X/S}^{0}$ and $\mathbf{Pic}_{X/S}^{\tau}$ as open and closed subgroup schemes of $\mathbf{Pic}_{X/S}$. Moreover, we assume $\mathbf{Pic}_{X/S}^{0}$ is a projective abelian scheme over $S$ and $\mathbf{Pic}_{X/S}^{\tau}$ is projective and flat over $S$. 
\end{nota}

We fix some further conventions: if $G$ is any $S$-group scheme and $H$ a closed subgroup scheme of $G$ we denote by  $G/H_{(fpqc)}$ the sheaf associated, with respect to the fpqc topology, to the functor $$T\mapsto G(T)/H(T)$$ from the category of schemes over $S$ to the category of sets.  Now we recall a particular case of \cite{Gab}, Th\'eor\`eme 7.1:

\begin{teo}\label{teoRappre1}
Let $G$ be a $S$-group scheme of finite type and let $H$ be a closed subgroup scheme of $G$. If $H$ is proper and flat over $S$ and if $G$ is quasi projective over $S$ then  $G/H_{(fpqc)}$ is representable.
\end{teo}

Now, $\mathbf{Pic}_{X/S}^{0}$ is an open and closed subgroup scheme of $\mathbf{Pic}_{X/S}^{\tau}$ then we  consider the quotient sheaf
$${NS_{X/S}^{\tau}}:={\mathbf{Pic}_{X/S}^{\tau}/\mathbf{Pic}_{X/S}^{0}}_{(fpqc)}.$$

\begin{prop}\label{propNeronEsiste} The sheaf ${NS_{X/S}^{\tau}}$ is represented by a separated and flat $S$-group scheme of finite type and forming it commutes with changing the base; furthermore the projection $p:\mathbf{Pic}_{X/S}^{\tau}\to \mathbf{NS}_{X/S}^{\tau}$ is faithfully flat.
\end{prop}
\proof It exists by theorem \ref{teoRappre1},   then it is a $S$-group scheme, separated, flat and commutes with base changing according to \cite{BER}, Proposition 9.2, respectively (iv), (x), (xi) and (v).\endproof

\begin{defin}
We denote by $\mathbf{NS}_{X/S}^{\tau}$ the $S$-scheme representing ${NS_{X/S}^{\tau}}$ and we call it the torsion Néron-Severi scheme.
\end{defin}

\begin{prop} The torsion Néron-Severi scheme $\mathbf{NS}_{X/S}^{\tau}$ is finite and smooth over $S$. 
\end{prop}
\proof It is proper according to \cite{Liu}, Proposition 3.3.16 (f), since it is separated and of finite type over $S$, $p:\mathbf{Pic}_{X/S}^{\tau}\to \mathbf{NS}_{X/S}^{\tau}$ is surjective and  $\mathbf{Pic}_{X/S}^{\tau}$ is proper over $S$. For any  $s\in S$ then $\mathbf{NS}_{X_{k(s)}/{k(s)}}^{\tau}$ is the \'etale finite (then quasi finite) ${k(s)}$-group scheme of connected components $\pi_{0}(\mathbf{Pic}_{X_{k(s)}/{k(s)}}^{\tau})$ (cf. \cite{DemGab}, II, \S 5, n$^{\circ}$ 1). Thus $\mathbf{NS}^{\tau}_{X/S}$, being flat over $S$, is smooth. Since it is proper and quasi finite then it is moreover finite by \cite{Mil}, ch. 1, \S 1, Corollary 1.10.\endproof

Now consider the following diagram where horizontal arrows are exact sequences, $n>0$ is any integer, the first vertical line is an isogeny since $\mathbf{Pic}_{X/S}^{0}$ is an abelian scheme and everywhere we denote by ${}_nG$ the kernel $ker(n:G\to G)$ for any commutative group scheme $G$:

$$\xymatrix{ & 0 \ar[d]& 0 \ar[d]& 0 \ar[d]&\\& {}_n\mathbf{Pic}_{X/S}^{0} \ar[d] & {}_n\mathbf{Pic}_{X/S}^{\tau} \ar[d] & {}_n\mathbf{NS}_{X/S}^{\tau} \ar[d] & \\ 
0 \ar[r] & \mathbf{Pic}_{X/S}^{0} \ar[r] \ar[d]^{n} & \mathbf{Pic}_{X/S}^{\tau} \ar[r]\ar[d]^{n} &  \mathbf{NS}_{X/S}^{\tau} \ar[r]\ar[d]^{n}& 0 \\
0 \ar[r] & \mathbf{Pic}_{X/S}^{0} \ar[r] \ar[d] & \mathbf{Pic}_{X/S}^{\tau} \ar[r] &  \mathbf{NS}_{X/S}^{\tau} \ar[r]& 0 \\ & 0 & & &}$$

\noindent by the snake lemma (applied if we place ourselves in the abelian category of fpqc sheaves) we obtain the following exact sequence of finite commutative $S$-group schemes:

\begin{equation}\label{eqSuite1}\xymatrix{ 0\ar[r]& {}_n\mathbf{Pic}_{X/S}^{0} \ar[r] & {}_n\mathbf{Pic}_{X/S}^{\tau} \ar[r] & {}_n\mathbf{NS}_{X/S}^{\tau} \ar[r] & 0}\end{equation} 

Indeed ${}_n\mathbf{Pic}_{X/S}^{0}$ is faithfully flat over $S$, thus the projection ${}_n\mathbf{Pic}_{X/S}^{\tau}\to \mathbf{NS}_{X/S}^{\tau}$ is faithfully flat (\cite{BER}, Proposition 9.2, (xi)), so the sequence is an exact sequence of group schemes (\cite{Oort}, III, Property 17.1). Moreover ${}_n\mathbf{Pic}_{X/S}^{\tau}$ is finite since ${}_n\mathbf{Pic}_{X/S}^{0}$ and ${}_n\mathbf{NS}_{X/S}^{\tau}$ are finite, cf. \cite{BER}, Proposition 9.2, (viii)

\begin{rem}\label{ossNeronTorsione} Since $\mathbf{NS}_{X/S}^{\tau}$ is finite, let $N$ be its order, then for any positive integer $m$ we have 
${}_{N\cdot m}\mathbf{NS}_{X/S}^{\tau}=\mathbf{NS}_{X/S}^{\tau}$ since $(N\cdot m)\cdot 1_{\mathbf{NS}_{X/S}^{\tau}}$ is the zero map thus we have an exact sequence of finite and flat $S$-group schemes
\begin{equation}\label{eqSuiteFlat}\xymatrix{ 0\ar[r]& {}_{N\cdot m}\mathbf{Pic}_{X/S}^{0} \ar[r] & {}_{N\cdot m}\mathbf{Pic}_{X/S}^{\tau} \ar[r] & \mathbf{NS}_{X/S}^{\tau} \ar[r] & 0.}\end{equation}

\end{rem}
Now, dualizing the isogeny of previous diagram we obtain, by theorem \ref{teoDualSequence}, the following exact sequence

\begin{equation}\label{eqSuiteAbel}\xymatrix{ 0\ar[r]& ({}_n\mathbf{Pic}_{X/S}^{0})^{\vee} \ar[r] & \mathbf{Alb}_{X/S} \ar[r]^{n} & \mathbf{Alb}_{X/S} \ar[r] & 0}\end{equation} 
whence  \begin{equation}\label{eqIsoAbel}{}_n\mathbf{Alb}_{X/S}\simeq ({}_n\mathbf{Pic}_{X/S}^{0})^{\vee}.\end{equation}

\section{The abelian fundamental group scheme}
In \cite{Nor}, Nori defines the fundamental group scheme
$\pi_1(X,x)$ of a reduced, connected and proper scheme $X$ over a
perfect field $k$ provided with a rational point $x\in X(k)$ as the group
scheme associated to the neutral tannakian category
$(EF(X),\otimes,x^{\ast},\mathcal{O}_X)$ of essentially finite vector bundles. As pointed out in \cite{AnEm}, \S 1, the same construction holds over any base field $k$ with the additional assumption $H^0(X,\mathcal{O}_X)=k$. Nori in particular has shown  that the category of
torsors over $X$, under the action of finite group schemes, pointed above $x$ is filtered, and that
the fundamental group scheme $\pi_1(X, x)$ is the projective limit of the group schemes
occurring in these torsors. This led Nori to generalize the notion of fundamental
group scheme: a $k$-scheme $X$ pointed at
$x \in X(k)$ has a fundamental group scheme based at x if there exists a universal
torsor pointed above $x$ that dominates every torsor pointed above $x$ under the
action of a finite group scheme (cf. \cite{Nor2}, Chapter II, Definition 1).  Then Nori 
proves that this is equivalent as saying that the category of torsors under finite
group schemes over $X$ pointed above $x$ is filtered (cf. \cite{Nor2}, Chapter II, Poposition
1). This point of view has been generalized by Gasbarri who constructs in \cite{Gas} the fundamental group scheme $\pi_1(X, x)$ of an integral scheme $X$ over a connected Dedekind scheme $S$ provided with a $S$-valued point $x\in X(S)$ as the projective limit of the finite and flat group schemes occurring in torsors over $X$ pointed above $x$.

If instead of considering all torsors we only consider commutative torsors (i.e. torsors under the action of a finite, flat and commutative group scheme) then the projective limit of the group schemes occurring in them is called the abelianized (or simply \textit{abelian}) fundamental group scheme and denoted by $\pi_1^{ab}(X,x)$. There is a canonical morphism  $$\pi_1(X,x)\to\pi_1^{ab}(X,x),$$ which is not, in general, an isomorphism. However by a result of Nori it can sometimes be an isomorphism:

\begin{prop}\label{propGrupFondAbel} Let $S$ be a connected Dedekind scheme, $X\to S$ an abelian scheme and $0_X\in X(S)$ the identity element. Then its fundamental group scheme $\pi_1(X,0_X)$ is isomorphic to $\underleftarrow{lim}_{n\in\mathbb{N}}({}_n X)$ and then commutative.
\end{prop}
\proof When $S$ is the spectrum of a field the result has been proved in \cite{Nor3} by Nori. The same is true when $S$ is a Dedekind scheme and $X$ an abelian $S$-scheme (cf. \cite{Antei}, \S 2.2). 
\endproof

In the next section we will study the $S$-group scheme $\pi_1^{ab}(X,x)$ for a quite general scheme $X$ over a Dedekind scheme.

\subsection{Commutative finite torsors}
\label{sez:Commutative}
Let $S$ be a connected Dedekind scheme, $X$ an integral scheme, $f:X\to S$ a faithfully flat morphism of finite type provided with a section $x:S\to X$ and such that $\mathcal{O}_S\simeq f_{\ast}(\mathcal{O}_X)$ holds universally. Under these assumptions the fundamental group scheme $\pi_1(X,x)$ is always defined. In section \ref{sez:Picard} we have discussed some existence theorems for $\mathbf{Pic}_{X/S}$, $\mathbf{Pic}_{X/S}^{\tau}$ and $\mathbf{Pic}_{X/S}^{0}$. However throughout this section we simply assume they exist and they satisfy all the conditions of  notation \ref{notaNS}. If $G$ is a finite, flat and commutative $S$-group scheme we consider $H^1(X,G):=H^1_{fpqc}(X,G)$ the group of isomorphism classes of $G$-torsors over $X$. Let $H^1_{\bullet}(X,G)$ denote the subgroup of $H^1(X,G)$ of isomorphism classes of $G$-torsors $Y$ over $X$ provided with a $S$-valued point $y\in Y_x(S)$. We prove the following proposition, implicitly suggested in \cite{SGA1}, XI (last page):

\begin{prop}\label{propCorrispondenza} Let $G$ be a finite, flat and commutative $S$-group scheme. Then there exists a group isomorphism between $H^1_{\bullet}(X,G)$ and\\ $Hom_S(G^{\vee}, \mathbf{Pic}_{X/S}^{\tau})$.\end{prop}

\noindent Proposition \ref{propCorrispondenza} actually holds over any base scheme $S$. In order to prove it we first need a lemma: 

\begin{lem}\label{lemmaGrazieMichel} Let $G$ be a finite, flat and commutative $S$-group scheme. Then the natural inclusion\footnote{Note that $H^1_{\bullet}(S,G)=\{0\}$.} $H^1_{\bullet}(X,G)\hookrightarrow H^1(X,G)/H^1(S,G)$ is a group isomorphism.
\end{lem}
\proof We recall the multiplication law in $H^1(X,G)$ (more details can be found in \cite{DemGab}, III, \S 4, n$^{\circ}$ 3 and 4): let $Y, T\in H^1(X,G)$ then the (commutative) product $Y\cdot T$ of the two torsors is defined as the contracted product $Y\times_X^G T$. Note that $H^1(S,G)$ can be seen as the group of $G$-torsors over $X$ which are pull back of $G$-torsors over $S$. Now set $T:=Y_x\times_S X$, then clearly $T\in H^1(S,G)$. We follow \cite{Sko}, 2.2, page 22, (replacing $Spec(k)$ by $S$): the contracted product commutes with base change (cf. \cite{DemGab}, III, \S 4, n$^{\circ}$ 3, 3.1) then $(Y\times_X^G T)_{x}\simeq Y_x\times_S^G Y_x$. From the diagonal embedding $Y_x\hookrightarrow Y_x\times_S Y_x$ one deduces the closed immersion $$S=Y_x/G\hookrightarrow (Y_x\times_S Y_x)/G=Y_x\times_S^G Y_x=(Y\cdot T)_x.$$ In conclusion we have found a torsor $T\in H^1(S,G)$ such that $Y\cdot T\in H^1_{\bullet}(X,G)$ and this is enough to conclude.\endproof

\begin{proof}[Proof of Proposition] \ref{propCorrispondenza}. (see also \cite{AnGa}, 5.12 for a similar construction) 
Let  $q:Y\to X$ be a $G$-torsor and let $T$ be any $S$-scheme, we want to associate it a  morphism from $G^{\vee}(T)$ to $\mathbf{Pic}_{X/S}^{\tau}(T)$. We use the Cartier-Shatz formula $$G^{\vee}(T)\simeq Hom_T(G_T,{\mathbb{G}_{m}}_T)$$ (see for example \cite{Oort}, III.16, Theorem 16.1 but also \cite{WW} \S 2.4) where ${\mathbb{G}_{m}}_T$ is the multiplicative group scheme over $T$. So let $g\in G^{\vee}(T)$ and let $\gamma_x\in Hom_T(G_T,{\mathbb{G}_{m}}_T)$ be the corresponding morphism. Simply by base change $q_T:Y_T\to X_T$ is a $G_T$-torsor and if we consider the contracted product $$Y^{g}:=Y_T\times^{G_T} {\mathbb{G}_{m}}_T$$  through the morphism $\gamma_x$ we have then found a ${\mathbb{G}_{m}}_T$-torsor over $X_T$, thus an element of $H^1(X_T,{\mathbb{G}_{m}}_T)  \simeq Pic(X_T)$ (cf. \cite{Mil}, III, Proposition 4.9). Thus to $g$ we have associated an invertible sheaf $\mathcal{F}^{g}\in Pic(X_T)$; we denote by $\overline{\mathcal{F}^{g}}$ its image in $Pic(X_T)/Pic(T)\simeq \mathbf{Pic}_{X/S}(T)$ (cf. theorem \ref{teoKleiman}). Then we have constructed a morphism $$H^{1}(X,G)\to Hom_S(G^{\vee},\mathbf{Pic}_{X/S})$$ which passes to quotient $$\omega: H^{1}(X,G)/H^{1}(S,G)\to Hom_S(G^{\vee},\mathbf{Pic}_{X/S});$$ it is an isomorphism according to \cite{Mil}, III, Proposition 4.16. Moreover since $G^{\vee}$ is finite of order $m:=|G|$, thus $m:G^{\vee}\to G^{\vee}$ is the zero map so $Hom_S(G^{\vee},\mathbf{Pic}_{X/S})\simeq Hom_S(G^{\vee},\mathbf{Pic}^{\tau}_{X/S})$ and $\omega$ then factors through $$Hom_S(G^{\vee},\mathbf{Pic}^{\tau}_{X/S}).$$  Using lemma \ref{lemmaGrazieMichel} we finally obtain the isomorphism $$\hat{\omega}: H^1_{\bullet}(X,G)\simeq Hom_S(G^{\vee},\mathbf{Pic}_{X/S}^{\tau}).$$ \end{proof}

The previous proposition and what follows generalize what Nori did for abelian varieties in \cite{Nor3}. Here we give some details that in \cite{Nor3} are left to the reader. So observe, following the proof of previous proposition, that a morphism $\varphi:G^{\vee}\to \mathbf{Pic}^{\tau}_{X/S}$ can still be factored through ${}_m\mathbf{Pic}^{\tau}_{X/S}$ where $m=|G|$, then through ${}_{N\cdot m}\mathbf{Pic}^{\tau}_{X/S}$, where $N=|\mathbf{NS}^{\tau}_{X/S}|$, via the inclusion ${}_m\mathbf{Pic}^{\tau}_{X/S}\hookrightarrow {}_{N\cdot m}\mathbf{Pic}^{\tau}_{X/S}$, so in particular we have an isomorphism $$H^1_{\bullet}(X,G)\simeq Hom_S(G^{\vee},{}_{N\cdot m}\mathbf{Pic}_{X/S}^{\tau})$$ and consequently $$ H^1_{\bullet}(X,G)\simeq Hom_S(({}_{N\cdot m}\mathbf{Pic}_{X/S}^{\tau})^{\vee},G).$$
Let us give an interpretation to this last isomorphism. First of all we recall that $H^1_{\bullet}(X,G)\simeq Hom_S(\pi_1(X,x),G)\simeq Hom_S(\pi_1(X,x)^{ab},G)$ since $G$ is commutative, then we have 
$$(\star)\qquad\qquad Hom_S(\pi_1(X,x)^{ab},G)\simeq Hom_S(({}_{N\cdot m}\mathbf{Pic}_{X/S}^{\tau})^{\vee},G)$$
thus in particular 
$$Hom_S(\pi_1(X,x)^{ab},({}_{N\cdot m}\mathbf{Pic}_{X/S}^{\tau})^{\vee})\simeq Hom_S(({}_{N\cdot m}\mathbf{Pic}_{X/S}^{\tau})^{\vee},({}_{N\cdot m}\mathbf{Pic}_{X/S}^{\tau})^{\vee}).$$
Let $\rho_m^{\tau}:\pi_1(X,x)^{ab}\to ({}_{N\cdot m}\mathbf{Pic}_{X/S}^{\tau})^{\vee}$ be the morphism associated to $id_{({}_{N\cdot m}\mathbf{Pic}_{X/S}^{\tau})^{\vee}}$ thus we have a morphism $$\sigma:\pi_1(X,x)^{ab}\to \underleftarrow{lim}_{m\in\mathbb{N}}({}_{N\cdot m}\mathbf{Pic}_{X/S}^{\tau})^{\vee},$$ then using $(\star)$  we have the following
\begin{prop}\label{propLimitePicTau} Notation being as at the beginning of this section, the morphism $$\sigma:\pi_1(X,x)^{ab}\to \underleftarrow{lim}_{m\in\mathbb{N}}({}_{N\cdot m}\mathbf{Pic}_{X/S}^{\tau})^{\vee}$$ is an isomorphism. \end{prop} 

When $X$ is a projective abelian scheme over $S$ we recover (cf. Proposition \ref{propGrupFondAbel}) that $$\pi_1(X,x)^{ab}\simeq \underleftarrow{lim}_{m\in\mathbb{N}}({}_m X)$$ 
as $N=1$ and  $({}_m\mathbf{Pic}_{X/S}^{\tau})^{\vee}\simeq ({}_m\mathbf{Pic}_{X/S}^{0})^{\vee}\simeq {}_m X$, which gives in fact the whole fundamental group scheme.

\begin{lem}\label{lemmamorfFaithFlat} The natural morphism $\varphi^{ab}:\pi_1(X,x)^{ab}\to \pi_1(\mathbf{Alb}_{X/S},0_{\mathbf{Alb}_{X/S}})$ is faithfully flat.
\end{lem}
\proof All the morphisms $({}_{N\cdot n}\mathbf{Pic}_{X/S}^{\tau})^{\vee}$ $\to $ $({}_{N\cdot m}\mathbf{Pic}_{X/S}^{\tau})^{\vee}$ are faithfully flat for all positive (comparable) integers $n, m$ since the dual morphisms are closed immersions then the morphism $\rho^{\tau}_m:\pi_1(X,x)^{ab}\to ({}_{N\cdot m}\mathbf{Pic}_{X/S}^{\tau})^{\vee}$ is faithfully flat too. Similarly the canonical map $\rho^{0}_m:\pi_1(\mathbf{Alb}_{X/S},0_{\mathbf{Alb}_{X/S}})\to ({}_m\mathbf{Pic}_{X/S}^{0})^{\vee}$ is faithfully flat for every positive $m\in \mathbb{N}$. Now $\varphi^{ab}$ is faithfully flat if and only if for every $m$ the morhism $\rho^{0}_m\circ \varphi^{ab}$ is faithfully flat. But $\rho^{0}_m\circ \varphi^{ab}$ factors into $\rho^{\tau}_m$ (which is faithfully flat) and $({}_{N\cdot m}\mathbf{Pic}_{X/S}^{\tau})^{\vee}\to ({}_m\mathbf{Pic}_{X/S}^{0})^{\vee}$ which is faithfully flat too since the dual is a closed immersion (see the exact sequence (\ref{eqSuite1})). \endproof

\begin{teo}\label{teoUnico} We have the following exact sequence of commutative group schemes 
$$\xymatrix{0\ar[r] &(\mathbf{NS}^{\tau}_{X/S})^{\vee}\ar[r] &\pi_1(X,x)^{ab}\ar[r]^(.4){\varphi^{ab}} & \pi_1(\mathbf{Alb}_{X/S},0_{\mathbf{Alb}_{X/S}}) \ar[r] & 0.}$$ 
\end{teo}
\proof  According to lemma \ref{lemmamorfFaithFlat} we only need to prove that $ker (\varphi^{ab})\simeq (\mathbf{NS}^{\tau}_{X/S})^{\vee}$; so we dualize the exact sequence (\ref{eqSuiteFlat}) thus obtaining another exact sequence of finite and flat commutative $S$-group schemes:
$$\xymatrix{0 \ar[r] &(\mathbf{NS}_{X/S}^{\tau})^{\vee} \ar[r] & ({}_{N\cdot m}\mathbf{Pic}_{X/S}^{\tau})^{\vee} \ar[r] & ({}_{N\cdot m}\mathbf{Pic}_{X/S}^{0})^{\vee}  \ar[r]& 0 }$$ thus passing to the inverse limit we obtain

$$\xymatrix{0 \ar[r] & (\mathbf{NS}_{X/S}^{\tau})^{\vee} \ar[r] & \pi_1(X,x)^{ab} \ar[r]^(.4){\varphi^{ab}} &  \pi_1(\mathbf{Alb}_{X/S},0_{\mathbf{Alb}_{X/S}})   \ar[r]& 0}$$ 

\endproof

In next corollary we discuss the situation where every finite commutative pointed $G$-torsor over $X$ comes from its Albanese scheme (i.e. it is the pull back of  finite commutative pointed $G$-torsor over $\mathbf{Alb}_{X/S}$). 

\begin{cor}\label{corUnico1} When moreover $\mathbf{Pic}_{X/S}^{\tau}=\mathbf{Pic}_{X/S}^{0}$, then  $$\pi_1(X,x)^{ab} \simeq  \pi_1(\mathbf{Alb}_{X/S},0_{\mathbf{Alb}_{X/S}}).$$
\end{cor}
\proof Indeed in this case $\mathbf{NS}^{\tau}_{X/S}=S$, the trivial group scheme. \endproof

\subsection{Effective examples}
\label{sez:Effective}
Let $S$ be a Dedekind scheme, $X$ an integral scheme, $f:X\to S$ a faithfully flat morphism of finite type provided with a section $x:S\to X$ and such that $\mathcal{O}_S\simeq f_{\ast}(\mathcal{O}_X)$ holds universally. In section \ref{sez:Commutative} we have needed to assume the existence of $\mathbf{Pic}_{X/S}$, $\mathbf{Pic}_{X/S}^{0}$ and $\mathbf{Pic}_{X/S}^{\tau}$ satisfying certain conditions (cf. notation \ref{notaNS}) in order to prove in a more general possible setting our results. In this section  we will not  assume their existence anymore and we will describe some cases where they certainly exist and they are nice enough so that theorem \ref{teoUnico} can be successfully applied.

\medskip

\indent \textit{1. Curves}

\medskip

\noindent The following corollary generalizes a classical result stating that the abelianized \'etale fundamental group of a curve over a separably closed field is isomorphic to the \'etale fundamental group of its Jacobian (cf. \cite{Mil2}, \S 9). This is almost a restatement of 
theorem \ref{teoUnico} in the case of curves:
\begin{cor}\label{corCurve} Let $S$ be a Dedekind scheme and $f:C\to S$  a smooth and projective curve  with integral geometric fibers. Assume moreover the existence of a $S$-valued point $x\in C(S)$. Let $J:=J_{C/S}$ be the Jacobian of $C$ and $i:C\to J$ the canonical morphism sending  $x$ to $0_J\in J(S)$.  Then the natural morphism $\varphi:\pi_1(C,x)^{ab}\to \pi_1(J,0_J)$ of $S$-group schemes is an isomorphism.
\end{cor}
\proof This is a consequence of theorem \ref{teoPicardMumf} and corollary \ref{corUnico1}.
\endproof

\medskip

\indent \textit{2. Characteristic zero}\medskip

\noindent When $S$ is a Dedekind scheme of characteristic zero, then it is sufficient to take $f:X\to S$ projective and smooth. Then the existence of $\mathbf{Pic}_{X/S}$ is assured for example if $f$ has integral geometric fibers. In this case for every $s\in S$ all the $\mathbf{Pic}_{X_{k(s)}/k(s)}^0$ are smooth of the same dimension (cf. remark \ref{ossPicCarZero}) then use theorems  \ref{teoesistenza1} and \ref{teoesistenza2} to conclude. When $S$ is connected $\mathbf{Pic}^{\tau}_{X/S}$ is smooth since $\mathbf{Pic}_{X/S}$ is (cf. \cite{KL}, Remark 9.5.21).

\medskip

\indent \textit{3. Schemes over a complete d.v.r. of mixed characteristic}\medskip

\noindent  Let $S:=Spec(R)$ where  $R$ is a complete d.v.r. of mixed characteristic $(0,p)$ with field of fractions  $K$ and algebraically closed residue field $k$. Let $t$ be an uniformizing parameter of $R$, $v$ the valuation normalized by $v(t)=1$ and $e:=v(p)$ the absolute ramification index. Let $X\to S$ be a smooth and proper scheme and let $X_{\eta}$ and $X_{s}$ be respectively the generic and special fibers; then $\mathbf{Pic}_{X/R}$ exists (it is an algebraic space according to \cite{BLR}, \S 8.3, Theorem 1, thus a group scheme according to \cite{Ray2}, Th\'eor\`eme 3.3.1) and represents the relative Picard functor $Pic_{X/S}$. As $char(K)=0$ we have already seen that $\mathbf{Pic}_{X_{\eta}/K}^0$ is smooth; assume for a moment that $\mathbf{Pic}_{X_s/k}^0$ is smooth too then we still have to verify that $dim (\mathbf{Pic}_{X_{\eta}/K}^0)=dim (\mathbf{Pic}_{X_s/k}^0)$. This will be ensured by the following theorem due to Raynaud: 

\begin{teo}\label{teoRaunaud} Assume $e<p-1$, then $\mathbf{Pic}_{X/R}^0$ exists as a projective abelian scheme over $R$. Moreover $\mathbf{Pic}_{X/R}^{\tau}$ is a $R$-flat group scheme of finite type.
\end{teo}
\proof  Cf. \cite{Ray}, Th\'eor\`eme 1. Then use Theorem \ref{teoesistenza1} for the projectivity of $\mathbf{Pic}_{X/R}^0$ over $R$.\endproof

If we don't want to assume $\mathbf{Pic}_{X_s/k}^0$ to be smooth we can give sufficient and necessary conditions on $X$ implying the smoothness of $\mathbf{Pic}_{X_s/k}$ (then of $\mathbf{Pic}_{X_s/k}^0$, cf. \cite{DemGab}, II,\S 5, n$^{\circ}$ 2, 2.1). Indeed according to \cite{mumLec}, Lecture 27 (see also \cite{KL}, Remark 9.5.15)   $\mathbf{Pic}_{X_s/k}$ is smooth if and only if all the Serre's Bockstein operations $\beta_i$ vanish, where $$\beta_1:H^1(X_s, \mathcal{O}_{X_s})\to H^2(X_s, \mathcal{O}_{X_s}),\qquad \beta_i:ker(\beta_{i-1})\to coker(\beta_{i-1})\text{ for } i\geq 2$$
These operations always vanish, for instance, when $H^2(X_s, \mathcal{O}_{X_s})=\{0\}$.

\medskip

\indent \textit{4. Schemes over a field.}

\medskip

\noindent Let $S$ be the spectrum of a field $k$: $\mathbf{Pic}_{X/S}$ always exist when $f:X\to S$ is proper (cf. cf. \cite{Murre}, II.15, Theorem 2). The identity component $\mathbf{Pic}^0_{X/S}$ is a closed and open group subscheme of finite type of $\mathbf{Pic}_{X/S}$ and it is geometrically irreducible (cf. \cite{KL}, Proposition 9.5.3) and it is projective at least when $f$ is projective and smooth (cf. \cite{KL}, Exercise 9.5.7). The smoothness of $\mathbf{Pic}^0_{X/S}$ is clear in the zero characteristic case which has already been treated in point 2. When $char (k)>0$ one can add the Serre's Bockstein conditions (see point 3) to ensure the smoothness of $\mathbf{Pic}_{X/S}$. As for $\mathbf{Pic}^{\tau}_{X/S}$, by \cite{KL}, Proposition 9.6.12 it is an open closed group subscheme of finite type of $\mathbf{Pic}_{X/S}$ when $f:X\to S$ is projective and it is smooth (then flat) when the Serre's Bockstein conditions imply the smoothness of $\mathbf{Pic}_{X/S}$.

\subsection{Extension of commutative finite torsors}
\label{sez:Extension}

Let $S$ be a connected Dedekind scheme of generic point $\eta=Spec(K)$. Let $X$ be an integral scheme and $f:X\to S$ a faithfully flat morphism of finite type provided with a section $x\in X(S)$ and such that $f_{\ast}(\mathcal{O}_X)\simeq \mathcal{O}_S$ holds universally. Assume moreover that $\mathbf{Pic}_{X/S}$, $\mathbf{Pic}_{X/S}^{0}$ and $\mathbf{Pic}_{X/S}^{\tau}$ exist and that satisfy all the conditions of  notation \ref{notaNS} (in section \ref{sez:Effective} we have discussed some examples where this always happens). For any point $s\in S$  consider the  fiber $f_{s}:X_{s}\to Spec(k(s))$ which is naturally provided with a section $x_{s}\in X_{s}(k(s))$. According to proposition \ref{propLimitePicTau} we have the isomorphism
$$\pi_1(X,x)^{ab}\simeq \underleftarrow{lim}_{m\in\mathbb{N}}({}_{N\cdot m}\mathbf{Pic}_{X/S}^{\tau})^{\vee},$$
then clearly, since $\mathbf{Pic}_{X/S}^{\tau}$ commutes with base change, this implies, whenever $X_s$ has a fundamental group scheme, the isomorphism
$$(\pi_1(X,x)^{ab})_{s}\simeq \underleftarrow{lim}_{{N\cdot m}\in\mathbb{N}}({}_{N\cdot m}\mathbf{Pic}_{X_{s}/k(s)}^{\tau})^{\vee}\simeq \pi_1(X_{s},x_{s})^{ab}.$$
From now on we only consider the case $s=\eta$. Let $G$ be a finite and commutative $K$-group scheme. For a \textit{quotient} pointed $G$-torsor $q:Y\to X_{\eta}$ we mean a pointed torsor whose associated morphism $\pi_1(X_{\eta},x_{\eta})\to G$ is faithfully flat.  Using standard techniques (cf. for instance \cite{Antei}, lemma 2.7) we are able to prove the following 
  
\begin{teo}\label{teoEstensioneTorsoriCommutativi}  Let $G$ be a finite and commutative $K$-group scheme. Then every quotient pointed $G$-torsor $q:Y\to X_{\eta}$ can be extended to a pointed  $G'$-torsor $Z\to X$ for some (necessarily commutative) model $G'$ of $G$ finite and flat over $S$. 
\end{teo}
\proof
The morphism $\pi_1(X_{\eta},x_{\eta})\to G$ factors through $({}_{N\cdot m}\mathbf{Pic}_{X_{\eta}/k(\eta)}^{\tau})^{\vee}$, where $m=|G|$. So let denote by $u:({}_{N\cdot m}\mathbf{Pic}_{X_{\eta}/k(\eta)}^{\tau})^{\vee} \to G$ the canonical faithfully flat morphism coming from previous factorisation and set $D:=ker(u)$; then construct the scheme theoretic closure $\overline{D}$ of $D$ in $({}_{N\cdot m}\mathbf{Pic}_{X/S}^{\tau})^{\vee}$ and consider the quotient $G':=({}_{N\cdot m}\mathbf{Pic}_{X/S}^{\tau})^{\vee}/\overline{D}$, then we have the following diagram:
 $$\xymatrix{D\ar[r] \ar@{^{(}->}[d]& \overline{D} \ar@{^{(}->}[d]\\
 ({}_{N\cdot m}\mathbf{Pic}_{X_{\eta}/k(\eta)}^{\tau})^{\vee}\ar@{->>}[d]_{u}\ar[r] & ({}_{N\cdot m}\mathbf{Pic}_{X/S}^{\tau})^{\vee}\ar@{->>}[d]\\
 G\ar[r]\ar[d] & G'\ar[d]\\
 {\eta}\ar[r] & S,}$$ and composing $({}_{N\cdot m}\mathbf{Pic}_{X/S}^{\tau})^{\vee}\to G'$ with $\pi_1(X,x)\to ({}_{N\cdot m}\mathbf{Pic}_{X/S}^{\tau})^{\vee}$ we conclude.
\endproof

\medskip

\scriptsize

\begin{flushright} Marco Antei\\ 
Hausdorff Center for Mathematics\\Bonn, Germany\\
E-mail: \texttt{antei@math.univ-lille1.fr}\\ \texttt{marco.antei@gmail.com}\\
\end{flushright}

\end{document}